\documentclass[11pt,oneside]{article}
\usepackage{authblk}
\usepackage{latexsym}
\usepackage{etoolbox}
\usepackage[arrow,matrix]{xy}
\usepackage{stmaryrd}
\usepackage{amsfonts}
\usepackage{amsmath, amssymb,amscd,amsthm, amsxtra, bbm, mathrsfs}
\usepackage{cases}
\usepackage{cite}
\usepackage{array}
\usepackage{fancyhdr}
\usepackage{ifthen}
\usepackage{verbatim}
\usepackage[T1]{fontenc}
\usepackage{lmodern}
\usepackage{graphics}
\usepackage{setspace}
\usepackage{hyperref}
\usepackage{tikz}
\usepackage{tikz-cd}
\usepackage{pdflscape}

\usepackage[top=2.5cm, bottom=2.5cm, left=2cm, right=2.7cm]{geometry}

\theoremstyle{plain}
\patchcmd{\section}{\centering}{\raggedright}{}{}
\patchcmd{\section}{\normalfont}{\bfseries}{}{}
\patchcmd{\subsection}{\normalfont}{\bfseries}{}{}

\theoremstyle{plain}
\newtheorem{theorem}{Theorem}[section]
\newtheorem{corollary}[theorem]{Corollary}
\newtheorem{lem}[theorem]{Lemma}

\theoremstyle{definition}
\newtheorem{defi}[theorem]{Definition}

\numberwithin{figure}{section}

\theoremstyle{remark}

\numberwithin{equation}{section}
\usepackage{enumitem}
\setlist[enumerate]{left=0pt}
\title{On main eigenvalues of zero-divisor graphs of reduced rings}

\author[1]{Sakshi Jain\thanks{ sakshijain.ah@gmail.com}}
\author[1]{Y. M. Borse\thanks{ ymborse11@gmail.com}}
\author[1]{R. Barabde\thanks{ rushikeshbarbde@gmail.com}}

\affil[1]{Department of Mathematics, Savitribai Phule Pune University, Pune-411007, India}
\begin{document}
	\maketitle

	
	
	
	

	\begin{abstract}
		The problem of characterizing graphs with a prescribed number of main eigenvalues is a long-standing problem in spectral graph theory. Although some constructions are known, only a few produce infinite families of simple connected graphs with exactly $s \ge 2$ main eigenvalues. Zero-divisor graphs form a well-structured class of algebraic graphs whose spectra can be described explicitly using equitable partitions, making them a convenient setting to study main eigenvalues. In this paper, we prove that the zero-divisor graphs of reduced rings provide an infinite family of simple connected graphs with exactly $s$ main eigenvalues, and that certain induced bipartite subgraphs also have exactly $s$ main eigenvalues for any positive integer $s$.

		
		
		
	\end{abstract}
	
	
	\section{Introduction}
	\noindent
	The eigenvalues of a graph are the eigenvalues of its adjacency matrix. 
	An eigenvalue of a graph is called a \emph{main eigenvalue} if its eigenspace is not orthogonal to the all-one vector $\mathbf{e}$, and a \emph{non-main eigenvalue} otherwise. By the main (respectively, non-main) eigenvalues of a graph, we mean the distinct main (respectively, non-main) eigenvalues. 
	The study of main eigenvalues of a graph has been an active area of research in spectral graph theory since its introduction by Cvetkovi\'c~\cite{cvetkovic2} in 1978.

	
	\par The problem proposed by Cvetković~\cite{cvetkovic2} of characterizing graphs with a prescribed number of main eigenvalues is a long-standing problem and has attracted considerable attention over the years. It follows from the well-known Perron-Frobenius theorem that every connected graph on at least two vertices has at least one main eigenvalue. Graphs with exactly one main eigenvalue are precisely the class of regular graphs \cite{cvetkovic2}. The graphs with exactly two main eigenvalues have been studied in a series of papers \cite{hagos, rowlinson, hou1, hou2, hou3, ghebleh}. However, despite extensive study, a complete characterization of such graphs has yet to be achieved. For graphs with three main eigenvalues, some families of graphs have been constructed in \cite{franca2026}. At the other end of the spectrum, graphs in which all eigenvalues are main, known as controllable graphs, were studied by Godsil\cite{godsil}. Also, graphs with $n-1$ main eigenvalues have been considered more recently by Du et al. \cite{du1,du2}. 
	
	In general, the characterization of graphs with exactly \( s \ge 2 \) main eigenvalues remains largely open. Hagos \cite{hagos} proved the number of main eigenvalues of a graph is equal to the rank of its \textit{walk matrix}. The \textit{walk matrix} of a graph on $n$ vertices with adjacency matrix $A$ is an $n \times n$ matrix with $n$ columns $\mathbf{e}, A\mathbf{e}, A^2\mathbf{e}, \dots, A^{n-1}\mathbf{e}.$ Huang~\cite{huang} extended this result by proving that the number of main eigenvalues of a graph equals the rank of the walk matrix of any quotient of the graph. They further proposed a method to construct connected graphs with exactly s main eigenvalues. However, this construction produces non-simple graphs and fails to yield infinite families of simple graphs.

	
	In this paper, we provide two infinite families of  graphs arising from zero-divisor graphs of certain classes of reduced rings, including Boolean rings, with exactly \(s\) main eigenvalues. Zero-divisor graphs are an important class of algebraic graphs that link ring-theoretic structure with combinatorial and spectral properties. For a commutative ring $R$ with unity, the zero-divisor graph $\Gamma(R)$ is defined as the simple graph whose vertex set consists of all nonzero zero-divisors of $R$, with two distinct vertices $x$ and $y$ adjacent if and only if $xy=0$. This graph was introduced by Anderson and Livingston \cite{anderson} as a modification of the earlier construction of Beck \cite{beck}. Over the years, zero-divisor graphs have attracted significant attention due to their rich structural and spectral properties.
	
	A systematic spectral study of such graphs was initiated by LaGrange for Boolean rings \( \mathbb{Z}_2^n \) in a series of papers \cite{lagrange1,lagrange2,lagrange3,lagrange4 }. LaGrange\cite{lagrange4} established the reciprocal eigenvalue property of the Boolean graph which later on, Kadu, Sonawane and Borse \cite{kadu} studied in a more general setting through the zero-divisor graphs of subposets of the Boolean algebra \( B_n \). LaGrange  constructed two Pascal-type matrices \(P\) and \(Q\), each of order \(n-1\), and expressed the spectrum of this graph in terms of the spectra of these matrices \cite{lagrange1,lagrange2}. The matrix \(P\) is the adjacency matrix of the quotient of the Boolean graph with respect to a particular equitable partition. Similarly, \(Q\) is the adjacency matrix of a quotient of a certain induced bipartite subgraph of this graph. Sonawane, Kadu and Borse \cite{sonawane1} extended this study to the zero-divisor graph \( \Gamma(R_n) \) of the reduced ring
	$R_n := {\mathbb{F}_m \times \mathbb{F}_m \times \cdots \times \mathbb{F}_m}~~~$(with $n$ terms),  where \( \mathbb{F}_m \) is a finite field of order $m$ and $n$ is an integer greater than 1. 
	The adjacency matrix of this graph has order \( m^n - (m-1)^n - 1 \). An equitable partition of this graph yields a Pascal-type matrix \( P[m,n] \) of order \( n-1 \). Similarly, an equitable partition of a particular vertex-induced bipartite subgraph yields another associated matrix \( Q[m,n] \) of order \( n-1 \) having eigenvalues $-\phi^i\xi^{n-i},$ where $\phi= \frac{1+\sqrt{4m -3}}{2}$ is the generalized golden ratio and $\xi= -(m-1)\phi^{-1},$ for $i \in \{1, 2, \dots, n-1\}.$  Sonawane~\cite{sonawane2} showed that the eigenvalues of \(P[m,n]\) and \(Q[m,n]\) are simple and nonzero, their spectra are disjoint, and together these spectra determine the spectrum of the graph \(\Gamma(R_n)\).
	\vskip.4cm
	\begin{theorem}[Sonawane~\cite{sonawane2}]\label{spectS}
		Let $m \ge 2$ and $n \ge 2$. The eigenvalues of the graph $\Gamma(R_n)$ consist precisely of the eigenvalues of $P[m,n]$, together with those of $-Q[m,n]$ and $0$. Moreover, every eigenvalue of $P[m,n]$ is simple as an eigenvalue of $\Gamma(R_n)$, and each eigenvalue  $\phi^i \xi^{\,n-i}$ of $-Q[m,n]$ occurs as an eigenvalue of $\Gamma(R_n)$ with multiplicity $\binom{n}{i}-1$, for $i \in \{1, 2, \dots, n-1\}$.
	\end{theorem}
	Lagrange~\cite{lagrange2} proved that for the Boolean graph (that is, $R_n$ with $m=2$), the above description remains valid except that $0$ does not occur as an eigenvalue.
	
	In this work, we prove that all \(n-1\) eigenvalues of \(P[m,n]\) are main eigenvalues of \(\Gamma(R_n)\), and hence they are precisely the main eigenvalues of the graph. Similarly, we show that the eigenvalues of \(Q[m,n]\) are exactly the main eigenvalues of its corresponding bipartite subgraph, and that they are the negatives of the nonzero, non-main eigenvalues of the original graph. Consequently, for any integer \(s \ge 2\), by taking \(n = s+1\) and varying \(m\), we obtain an infinite family of simple connected graphs \(\Gamma(R_{s+1})\) having exactly \(s\) main eigenvalues and \(s\) nonzero non-main eigenvalues. At the same time, this yields an infinite family of bipartite vertex-induced subgraphs whose main eigenvalues are precisely these non-main eigenvalues (up to sign) of the original graphs.
	\newpage
	Our proof is based on computing the rank of the walk matrices associated with \( P[m,n] \) and \( Q[m,n] \). We express these walk matrices in terms of a Vandermonde matrix formed from a sequence of generalized Fibonacci numbers \( F_{m,k} \). 
	
	\section{Preliminary}
	
	\noindent
	We recall some basic definitions and results used throughout the paper. 
	For a positive integer $k$, let $[k] := \{1,2,\dots,k\}$. For a positive integer $n$, let $\mathbf{e}_i$, for $i \in [n]$, denote the vector in $\mathbb{R}^n$ with $1$ in the $i$th coordinate and $0$ in all remaining coordinates.
	For a graph $G$, the vertex partition $\Pi : V(G) = V_1 \cup V_2 \cup \dots \cup V_r$ is said to be an equitable partition of $G$ if any vertex $v \in V_i$ has exactly $b_{ij}$ neighbors in $V_j$, for every pair $i,j \in [r]$. Given an equitable partition $\Pi$ of $G,$ a quotient $G/\Pi$  of $G$ is a directed multigraph with vertex set $\{ V_1,V_2, \dots, V_r \}$ and $b_{ij}$ arcs from $V_i$ to $V_j.$ Then, the adjacency matrix of the quotient $G/\Pi,$ denoted by $B_{\Pi}$, is an $r \times r$ matrix with  $(i, j)$-entry as the number of arcs from $V_i$ to $V_j$, that is,  $B_{\Pi} = (b_{ij})_{r\times r}.$  This matrix is called the \textit{quotient matrix} of $G$ corresponding to the partition $\Pi.$ By eigenvalues of the quotient $G/\Pi,$ we mean eigenvalues of the matrix $B_{\Pi}.$ 
	
	The \textit{walk matrix} of the quotient $G/\Pi$ is an $r\times r$ matrix denoted by $W(G/\Pi)$ or $W(B)$ and is defined as $W(B) = [\mathbf{e}, B\mathbf{e}, B^2\mathbf{e}, \dots, B^{r-1}\mathbf{e}],$ where $\mathbf{e}$ is the all-one vector of size $r.$ 
	Then the $(i, k + 1)$-entry of the walk matrix is given by $W(B)_{i,k+1} =[B^k\mathbf{e}]_i $ and it counts the number of walks of length $k$ starting from a vertex $v\in V_i$ \cite{rowlinson, harary}. Thus, any vertex in a cell $V_i$ of the equitable partition $\Pi$ possesses the same number of walks of length $k$ starting from it.
	
	Cvetković\cite{cvetkovic2}  proved a fundamental result that connects the main eigenvalues of a graph with the eigenvalues of its quotient matrices.
	
	\begin{lem}\textnormal{\cite{cvetkovic2}}\label{main eig of G is an eig of Quotient} 
		The spectrum of any quotient of a graph $G$ contains all main eigenvalues of $G.$
	\end{lem}

	Huang, Huang and Lu~\cite{huang} proved that the number of main eigenvalues of a graph equals the rank of the walk matrix of any of its quotient matrices.
	
	\begin{lem}{\normalfont\cite{huang}}\label{lem: main eigen. = rank of W(B)}
		Let $\Pi$ be an equitable partition of a graph $G$, and let $B$ be the quotient matrix of $G$ with respect to $\Pi$. Then the number of main eigenvalues of $G$ is equal to $\operatorname{rank}(W(B))$.
	\end{lem}
	
	We consider the determinant formula for a Vandermonde matrix, which will be used in the proof of the main theorems.
	
	\begin{defi}\cite{khare}
		Let $x_1, x_2, \dots, x_n$ be elements of a field $\mathbb{F}$. The \emph{Vandermonde matrix} associated with $\{x_1, x_2, \dots, x_n\}$ is defined (as the transpose of the usual Vandermonde form) by
		\[
		V=
		\begin{bmatrix}
			1 & 1 & \cdots & 1 \\
			x_1 & x_2 & \cdots & x_n \\
			x_1^{2} & x_2^{2} & \cdots & x_n^{2} \\
			\vdots & \vdots & \ddots & \vdots \\
			x_1^{\,n-1} & x_2^{\,n-1} & \cdots & x_n^{\,n-1}
		\end{bmatrix}.
		\]
		
		Thus, $
		V = \bigl(x_k^{\,i-1}\bigr)_{\substack{ i,k= 1}}^n$,
		where the index $i$ denote the rows and $k$ denote the columns.
	\end{defi}
	
	\begin{lem}\textnormal{\cite{khare}}\label{lem: vandermonde determinant}
		The determinant of the Vandermonde matrix $V$ is given by
		\[
		\det(V) = \prod_{1 \le i < j \le n} (x_j - x_i).
		\]
		In particular, $\det(V) \neq 0$ if and only if $x_i \neq x_j$ for all $i \neq j$.
	\end{lem} 
	\vskip.2cm
	\section{Generalized Fibonacci Numbers}
	
	\noindent
	For $k \ge 2$, the classical Fibonacci sequence $\{F_k\}$, with initial conditions $F_0 = 0$ and $F_1 = 1$, is defined by the recurrence relation
	\[
	F_k = F_{k-1} + F_{k-2}.
	\]
	
	For a fixed real number $\alpha$, a generalized Fibonacci sequence $\{F_{\alpha,k}\}$ is defined by
	\[
	F_{\alpha,k} = F_{\alpha,k-1} + (\alpha - 1)\,F_{\alpha,k-2}, \qquad k \ge 2,
	\]
	with initial conditions $F_{\alpha,0} = 0$ and $F_{\alpha,1} = 1$.
	
	We modify the sequence $\{F_{\alpha,k}\}$ by taking initial conditions as : $F_{\alpha,0} = 1$ and $F_{\alpha,1} = 1$. Using this modified sequence, we determine the number of walks of given length starting from a vertex in the zero-divisor graph $\Gamma(R_n).$  
	
	The following identity is a well-known generalization of Cassini's identity. 
	
	\begin{defi}\cite{koshy}[D’Ocagne’s Identity] 
		For integers $ l \geq r \geq 0,$ the Fibonacci numbers with initial conditions $F_0 = 0$ and $F_1 = 1$ satisfy the following identity: 
		
		\[ F_{l} F_{r+1} - F_{l+1}F_{r} = (-1)^rF_{l-r}.\]
		
	\end{defi}
	
	We  generalize this identity for the sequence $\{F_{\alpha,k} \}.$  
	
	\begin{lem}\label{docagnes identity}
		Let $\alpha $ be a real number and  $ l, r$ be integers with $ l>r \geq 0.$  The generalized Fibonacci numbers $F_{\alpha,k}$ with initial conditions $F_{\alpha, 0} = 1$ and $F_{\alpha, 1}= 1$ satisfy the identity
		\[
		F_{\alpha,l} F_{\alpha,r+1} - F_{\alpha,l+1} F_{\alpha,r}
		= (1-\alpha)^{r+1} F_{\alpha,l-r-1}
		\]
	\end{lem}
	
	\begin{proof} 
		Consider the recurrence relation $F_{\alpha,k} = F_{\alpha,k-1} + (\alpha-1)F_{\alpha,k-2},$ for $k \geq 2.$
		This recurrence relation can be written in the matrix form as
		\[
		\begin{bmatrix}
			F_{\alpha,k+1} \\
			F_{\alpha,k}
		\end{bmatrix}
		=
		\begin{bmatrix}
			1 & \alpha-1 \\
			1 & 0
		\end{bmatrix}
		\begin{bmatrix}
			F_{\alpha,k} \\
			F_{\alpha,k-1}
		\end{bmatrix}.
		\]
		Thus,
		\[
		\begin{bmatrix}
			F_{\alpha,k+1} \\
			F_{\alpha,k}
		\end{bmatrix}
		=
		\begin{bmatrix}
			1 & \alpha-1 \\
			1 & 0
		\end{bmatrix}^{n}
		\begin{bmatrix}
			F_{\alpha,k-n+1} \\
			F_{\alpha,k-n}
		\end{bmatrix}.
		\]
		
		Using this relation for two indices $l$ and $r,$ and taking $r$ iterations
		\[
		\begin{bmatrix}
			F_{\alpha,l+1} & F_{\alpha,r+1} \\
			F_{\alpha,l}   & F_{\alpha,r}
		\end{bmatrix}
		=
		\begin{bmatrix}
			1 & \alpha-1 \\
			1 & 0
		\end{bmatrix}^{r}
		\begin{bmatrix}
			F_{\alpha,l-r+1} & F_{\alpha,1} \\
			F_{\alpha,l-r}   & F_{\alpha,0}
		\end{bmatrix}.
		\]
		
		Since $F_{\alpha,0} = 1$ and $F_{\alpha,1} = 1$, this becomes
		\[
		\begin{bmatrix}
			F_{\alpha,l+1} & F_{\alpha,r+1} \\
			F_{\alpha,l}   & F_{\alpha,r}
		\end{bmatrix}
		=
		\begin{bmatrix}
			1 & \alpha-1 \\
			1 & 0
		\end{bmatrix}^{r}
		\begin{bmatrix}
			F_{\alpha,l-r+1} & 1 \\
			F_{\alpha,l-r}   & 1
		\end{bmatrix}.
		\]
		
		Taking determinants on both sides, we have
		\begin{align*}
			F_{\alpha,l+1}F_{\alpha,r} - F_{\alpha,l}F_{\alpha,r+1}
			&=
			\left(-(\alpha-1) \right)^r
			\cdot
			\left( F_{\alpha,l-r+1} - F_{\alpha,l-r} \right).
		\end{align*}

		Using the recurrence relation for $F_{\alpha,l-r+1},$ we have
		
		\[
		F_{\alpha,l-r+1} - F_{\alpha,l-r} = (\alpha-1)F_{\alpha,l-r-1}.
		\]
		
		Substituting back, we get the required identity
		\begin{align*}
			F_{\alpha,l+1}F_{\alpha,r} - F_{\alpha,l}F_{\alpha,r+1}
			&= (1-\alpha)^r \cdot (\alpha-1)F_{\alpha,l-r-1}\\
			F_{\alpha,l}F_{\alpha,r+1} - F_{\alpha,l+1}F_{\alpha,r}
			&=  (1-\alpha)^{r+1} F_{\alpha,l-r-1}.
		\end{align*}
		
	\end{proof}

	\section{Main spectrum of the zero-divisor graph \texorpdfstring{$\Gamma(R_n)$}{Gamma(Rn)}}
	
	\noindent
	In this section, we determine the number of main eigenvalues of the zero-divisor graph \(\Gamma(R_n)\), where 
	\( R_n = \mathbb{F}_m \times \mathbb{F}_m \times \cdots \times \mathbb{F}_m \) (with \(n\) terms) and \( \mathbb{F}_m \) is a finite field of order \(m\). 
	This is done by computing the rank of the walk matrix of the quotient matrix corresponding to a suitable equitable partition of the graph.

	The graph $\Gamma(R_n)$ has \( m^n-(m-1)^n-1 \) vertices, each corresponding to an \( n \)-tuple \( (a_1,a_2,\dots,a_n) \) with \( a_i\in\mathbb{F}_m \).  Sonawane et al. \cite{sonawane1} considered a particular  equitable partition of this graph  given by $\mathcal{D} = \{D_1, D_2, \dots, D_{n-1}\},$ where each cell $D_i$ contains vertices $v\in V(\Gamma(R_n))$ having exactly $i$  number of zero coordinates. The corresponding quotient matrix, denoted by $P[m,n],$ is a Pascal-type matrix of order $n-1$ defined as follows 
	
	\begin{defi}\cite{sonawane1}
		The quotient matrix $P[m,n]$ of the zero-divisor graph $\Gamma(R_n)$ with respect to the above equitable partition $\mathcal{D}$ is an $(n-1) \times (n-1)$ matrix, with $(i,j)$-entry given by
		\[
		P[m,n](i,j) =
		\begin{cases}
			\displaystyle \binom{i}{\,n-j\,}(m-1)^{\,n-j}, & \text{if } i + j \geq n, \\[6pt]
			\quad \quad 0, & \text{if } i + j < n.
		\end{cases}
		\]
	\end{defi}

	For example, 
	\[
	P[2,4] =
	\begin{bmatrix}
		0 & 0 & 1 \\
		0 & 1 & 2 \\
		1 & 3 & 3
	\end{bmatrix},
	\qquad
	P[3,4] =
	\begin{bmatrix}
		0 & 0 & 2 \\
		0 & 4 & 4 \\
		8 & 12 & 6
	\end{bmatrix}.
	\]
	
	Sonawane \cite{sonawane2} proved that all eigenvalues of the matrix \(  P[m,n]\) are simple and nonzero, and by Theorem \ref{spectS}, they are also simple eigenvalues of the graph $\Gamma(R_n).$
	
	Consider the walk matrix $W(P) = [\textbf{e},~ P\textbf{e},~ P^2\textbf{e},~ \dots,~ P^{n-2}\textbf{e}]$ of the matrix $P =P[m,n].$  The first column of this matrix is the all-one vector $\mathbf{e}.$  We now present an explicit formula for the entries of the $(k+1)$th column $P^k\mathbf{e}$ for each $ k \in [n-2],$
	in terms of the generalized Fibonacci numbers \(F_{m,k}\). To simplify the notation, we write \((F_{m,\ell})^{\,n}\) as \(F_{m,\ell}^{\,n}\).
	\begin{lem}\label{lem: entry of WP}
		Let $P = P[m,n]$ be the quotient matrix of the zero-divisor graph $\Gamma(R_n)$ corresponding to the equitable partition $\mathcal{D}$. Then, for $i \in [n-1]$ and $k \in [n-2]$, the $(i,k+1)$-entry of the walk matrix $W(P)$ is given by:

		\[
		[W(P)]_{i,k+1} = F_{m,k}^{\,n }  \gamma_{m,k}^i 
		- \sum_{j=0}^{k-1} h_j \, F_{m,k-j-1}^{\,n } \gamma_{m,k-j-1}^i\,,
		\]
		
		where,
		\[
		\gamma_{m,k} = \frac{F_{m,k+1}}{F_{m,k}}, 
		\qquad  h_0 =  1, \,\,
		h_j = F_{m,j+1}^{\,n }- \sum_{r=0}^{j-1} h_r \, F_{m,j-r}^{\,n }\,, \text{~for~} j\geq 1\, .
		\]
	\end{lem}
	
	\begin{proof} We prove the result by induction on $k$. Suppose $ k = 1,$ then the second column of the matrix $W(P)$ is $P\mathbf{e}.$  We compute its $i$th entry by evaluating the row sum:
		$$
		\left[ P \mathbf{e} \right]_i =\sum_{j=n-i}^{n-1} \binom{i}{n-j} (m-1)^{n-j}\,.
		$$
		With the change of variable $r = n - j$, we obtain
		$$
		[P\mathbf{e}]_i 
		= \sum_{r=1}^{i} \binom{i}{r} (m-1)^r 
		= (1 + (m-1))^i - 1 \\
		= \left( \frac{F_{m,1} + (m-1)F_{m,0}}{F_{m,1}} \right)^i - 1 \,.$$
		Since $F_{m,0} = 1 = F_{m,1}$ and $\gamma_{m,1} = \frac{F_{m,2}}{F_{m,1}}$, $\gamma_{m,0} = \frac{F_{m,1}}{F_{m,0}} = 1\,,$ thus by the recurrence formula $F_{m,k}$, we obtain
		\[
		[P\mathbf{e}]_i = \gamma_{m,1}^i - 1 = F_{m,1}^{\,n} \gamma_{m,1}^i - h_0 =  F_{m,1}^{\,n}\gamma_{m,1}^i - h_0\,F_{m,0}^{\,n}\gamma_{m,0}^i\,,
		\]
		which matches the required formula. Hence, the result holds for $k = 1.$ 
		

		Assume that the result holds for $ k.$ We prove it for $k+1$. The $(k+2)$th column of the matrix $W(P)$ is  $[P^{k+1}\mathbf{e}],$ which can be written as $P(P^k\mathbf{e}).$ 
		Using matrix multiplication, we compute its $i$th entry as follows:
		
		\[
		[W(P)]_{i,k+2} =  \left[ P^{k+1} \mathbf{e} \right]_i = \left[ P (P^k \mathbf{e}) \right]_i.
		\]
		
		Thus,
		\[
		\left[ W(P)\right]_{i,{k+2}}
		= \sum_{j=n-i}^{n-1} \binom{i}{n-j} (m-1)^{n-j} \,\left[ P^{k} \mathbf{e} \right]_j,
		\]
		\noindent
		changing variable $r = n - j$, we rewrite
		
		\begin{align*}
			\left[ P^{k+1} \mathbf{e} \right]_i
			&= \sum_{r=1}^{i} \binom{i}{r} (m-1)^r \, \left[ P^{k} \mathbf{e} \right]_{n-r}\\
			&= \sum_{r=1}^{i} \binom{i}{r} (m-1)^r
			\left(
			F_{m,k}^{\,n }  \gamma_{m,k}^{\,n-r} 
			- \sum_{j=0}^{k-1} h_j \, F_{m,k-j-1}^{\,n } \gamma_{m,k-j-1}^{\,n-r}
			\right)\\
			&= F_{m,k}^{\,n}  \gamma_{m,k}^{\,n}
			\sum_{r=1}^{i} \binom{i}{r} \left(\frac{m-1}{\gamma_{m,k}}\right)^r
			- \sum_{j=0}^{k-1} h_j \, F_{m,k-j-1}^{\,n}  \gamma_{m,k-j-1}^{\,n}
			\sum_{r=1}^{i} \binom{i}{r} \left(\frac{m-1}{\gamma_{m,k-j-1}}\right)^r\\
			&= F_{m,k}^{\,n}  \gamma_{m,k}^{\,n}
			\left[\left(1 + \frac{m-1}{\gamma_{m,k}}\right)^i - 1\right]
			\quad - \sum_{j=0}^{k-1} h_j \, F_{m,k-j-1}^{\,n}  \gamma_{m,k-j-1}^{\,n}
			\left[\left(1 + \frac{m-1}{\gamma_{m,k-j-1}}\right)^i - 1\right]\\
			&= F_{m,k+1}^{\,n}  \left( \gamma_{m,k+1}^i - 1 \right)
			- \sum_{j=0}^{k-1} h_j \, F_{m,k-j}^{\,n} \left( \gamma_{m,k-j}^i - 1 \right)\\
			&= F_{m,k+1}^{\,n} \gamma_{m,k+1}^i 
			- \sum_{j=0}^{k-1} h_j \, F_{m,k-j}^{\,n} \gamma_{m,k-j}^i  
			- \left(F_{m,k+1}^{\,n} - \sum_{j=0}^{k-1} h_j \, F_{m,k-j}^{\,n}\right)\\
			&= F_{m,k+1}^{\,n} \gamma_{m,k+1}^i 
			- \sum_{j=0}^{k} h_j \, F_{m,k-j}^{\,n} \gamma_{m,k-j}^i.
		\end{align*}

		
		
		which is precisely the required form.
		
		Thus, the statement holds for \(k+1\), and the proof is complete by induction.
	\end{proof}
	
	Note that the $(i,k + 1)$-entry of the walk matrix $W(P)$ equals the total number of walks of length $k$ in the original graph that start from any vertex in the cell $D_i$ of the equitable partition $\mathcal{D}$ \cite{rowlinson}. We now determine the rank of  $W(P)$ and thus find the number of main eigenvalues of the graph $\Gamma(R_n).$
	
	\begin{theorem}\label{thm:P-rank} The  zero-divisor graph $\Gamma(R_n)$ has exactly $n-1$ main eigenvalues. 
	\end{theorem}
	
	\begin{proof} Let $ P = P[m, n]$ be the quotient matrix corresponding to the equitable partition $\mathcal{D}$ of the graph $\Gamma(R_n).$ In view of Lemma \ref{lem: main eigen. = rank of W(B)}\,, it is sufficient to show that the rank of the walk matrix $W(P)$ is $n-1\,.$ Since $W(P)$ is an $(n-1)\times(n-1)$ matrix, we prove that $W(P)$ has a nonzero determinant.  
		
		Let $C_k = F_{m,k}^{\,n},$ for each $k\in\{0, 1, 2, \dots, n-2\}.$ Then $C_0 = F_{m,0}^n = 1.$  The first column of the walk matrix is $[1,\, 1,\, \dots, 1]^T,$ which can be replaced by the vector  $[C_0\gamma_{m,0},\, C_0\gamma_{m,0}^2,\, \dots,\, C_0\gamma_{m,0}^{n-1}]^T$ as $\gamma_{m,0} =   \frac{F_{m,1}}{F_{m,0}} = 1.$ This gives a Vandermonde-type structure to the matrix $W(P)$.
		By Lemma~\ref{lem: entry of WP}, this matrix has the form

		{\scriptsize
			\setlength{\arraycolsep}{3pt}
			\renewcommand{\arraystretch}{1.2}
			\[
			W(P) = 
			\begin{bmatrix}
				C_0\gamma_{m,0} 
				& C_1 \gamma_{m,1}- h_0C_0\gamma_{m,0} 
				& C_2 \gamma_{m,2} - \displaystyle\sum_{j=0}^{1} h_j C_{1-j}\gamma_{m,1-j} 
				& \cdots 
				& C_{n-2} \gamma_{m,n-2} - \displaystyle\sum_{j=0}^{n-3} h_j C_{n-j-3}\gamma_{m, n-j-3} \\
				
				C_0 \gamma_{m,0}^2 
				& C_1 \gamma_{m,1}^2 - h_0C_0\gamma_{m,0}^2 
				& C_2 \gamma_{m,2}^2 - \displaystyle\sum_{j=0}^{1} h_j C_{1-j}\gamma_{m,1-j}^2 
				& \cdots 
				& C_{n-2} \gamma_{m,n-2}^2 - \displaystyle\sum_{j=0}^{n-3} h_j C_{n-j-3}\gamma_{m, n-j-3}^2 \\
				
				C_0 \gamma_{m,0}^3 
				& C_1 \gamma_{m,1}^3 - h_0C_0\gamma_{m,0}^3 
				& C_2 \gamma_{m,2}^3 - \displaystyle\sum_{j=0}^{1} h_j C_{1-j}\gamma_{m,1-j}^3 
				& \cdots 
				& C_{n-2} \gamma_{m,n-2}^3 - \displaystyle\sum_{j=0}^{n-3} h_j C_{n-j-3}\gamma_{m, n-j-3}^3 \\
				
				\vdots & \vdots & \vdots & \ddots & \vdots \\
				
				C_0 \gamma_{m,0}^{n-1} 
				& C_1 \gamma_{m,1}^{n-1} - h_0C_0\gamma_{m,0}^{n-1} 
				& C_2 \gamma_{m,2}^{n-1} - \displaystyle\sum_{j=0}^{1} h_j C_{1-j}\gamma_{m,1-j}^{n-1} 
				& \cdots 
				& C_{n-2} \gamma_{m,n-2}^{n-1} - \displaystyle\sum_{j=0}^{n-3} h_j C_{n-j-3}\gamma_{m, n-j-3}^{n-1}
			\end{bmatrix}.
			\]
		}
		
		
		
		
		
		Observe that the columns of $W(P)$ can be written as a linear combination of the vectors $\mathbf{v}_0, \mathbf{v}_1,\dots, \mathbf{v}_{n-2},$ where  
		$\mathbf{v}_i = [C_i\gamma_{m,i},\, C_i\gamma_{m,i}^2,\, \dots,\, C_i\gamma_{m,i}^{n-1}]^T, \, \text{for } i \in \{0, 1, 2, \dots, n-2\}.$  Then the first column is $\mathbf{v}_0,$ second is $\mathbf{v}_1 - h_0\mathbf{v}_0,$ and in general, for $k \in [n-2]$, the $(k+ 1)$th column is 
		\begin{align*}
			P^k\mathbf{e} &= \mathbf{v}_k - h_0 \mathbf{v}_{k-1} - h_1 \mathbf{v}_{k-2} - \cdots - h_{k-2} \mathbf{v}_1 - h_{k-1} \mathbf{v}_0\\
			&= - h_{k-1} \mathbf{v}_0 - h_{k-2} \mathbf{v}_1 -\dots- - h_1 \mathbf{v}_{k-2} - h_0 \mathbf{v}_{k-1}+\mathbf{v}_k \\
			& =\begin{bmatrix}
				\mathbf{v}_0 & \mathbf{v}_1 & \cdots &  \mathbf{v}_{k-1} & \mathbf{v}_k & \mathbf{v}_{k+1}&\cdots& \mathbf{v}_{n-2} 
			\end{bmatrix}H_k\,,
		\end{align*}
		\noindent
		where $H_k$ is the coefficient vector of size $n-1$ as follows
		\[
		H_k = [-h_{k-1},\, -h_{k-2},\, \dots,\, -h_1,\, -h_0,\, 1,\, 0,\, \dots,\, 0\,]^T\,.
		\]
		Note that $H_{n-2} = [-h_{n-3},\, -h_{n-4},\, \dots,\, -h_1,\, -h_0,\, 1\,]^T\,.
		$

			Accordingly, $W(P)$ can be written as
			\[
			W(P) = B U\,,
			\]
			where $B$ is an $(n-1)\times(n-1)$ matrix formed by columns vectors $\mathbf{v}_i$ given by
			\begin{align*}
				\renewcommand{\arraystretch}{1.3}
				B &= \begin{bmatrix}
					\mathbf{v}_0 & \mathbf{v}_1 & \cdots & \mathbf{v}_{n-2} 
				\end{bmatrix}\\
				& =
				\begin{bmatrix}
					C_0\gamma_{m,0} & C_1\gamma_{m,1} & \cdots & C_{n-2}\gamma_{m,n-2} \\
					C_0\gamma_{m,0}^2 & C_1\gamma_{m,1}^2 & \cdots & C_{n-2}\gamma_{m,n-2}^2 \\
					\vdots & \vdots & \ddots & \vdots \\
					C_0\gamma_{m,0}^{n-1} & C_1\gamma_{m,1}^{n-1} & \cdots & C_{n-2}\gamma_{m,n-2}^{n-1}
				\end{bmatrix}
			\end{align*}
			
			and $U$ is an $(n-1)\times(n-1)$ upper triangular matrix of coefficients, explicitly of the form,
			
			\begin{align*}
				\renewcommand{\arraystretch}{1.3}
				U &= \begin{bmatrix}
					\mathbf{e}_1 & H_1 & H_2 & \cdots & H_{n-2} 
				\end{bmatrix}\\
				& =
				\begin{bmatrix}
					1 & -h_0 & -h_1 & \cdots & -h_{n-3} \\
					0 & 1 & -h_0 & \cdots & -h_{n-4} \\
					0 & 0 & 1 & \cdots & -h_{n-5} \\
					\vdots & \vdots & \vdots & \ddots & \vdots \\
					0 & 0 & 0 & \cdots & 1
				\end{bmatrix}\,. 
			\end{align*}
			
			Next, observe that $B$ can be factorized as
			\[
			B = V D\,,
			\]
			where
			\[\renewcommand{\arraystretch}{1.2}
			V =
			\begin{bmatrix}
				1 & 1 & 1 & \cdots & 1 \\
				\gamma_{m,0} & \gamma_{m,1}& \gamma_{m,2} & \cdots & \gamma_{m,n-2} \\
				\gamma_{m,0}^2 & \gamma_{m,1}^2 & \gamma_{m,2}^2 & \cdots & \gamma_{m,n-2}^2 \\
				\vdots & \vdots & \ddots & \vdots \\
				\gamma_{m,0}^{n-2} & \gamma_{m,1}^{n-2} &\gamma_{m,2}^{n-2} & \cdots & \gamma_{m,n-2}^{n-2}
			\end{bmatrix}
			\]
			is a Vandermonde matrix and
			$
			D = \mathrm{diag}(\,C_0 \gamma_{m,0},\, C_1 \gamma_{m,1},\, \cdots  ,\, C_{n-2} \gamma_{m,n-2}\,)
			$
			is a diagonal matrix.
			
			Thus,
			\[
			W(P) = V D \,U,
			\]
			\vskip.3cm
			\noindent and hence \, $\det(W(P)) = \det(V)\det(D)\det(U).$
			
			Since $U$ is upper triangular with all diagonal entries equal to $1$, we have $\det(U) = 1$. Also, $\det(D) \neq 0$ as for each $k \geq 0,$ we have $C_k = F_{m,k}^n \neq 0$ and hence, $\gamma_{m,k} = \frac{F_{m,k+1}}{F_{m,k}} \neq 0.$
			
			
			We now show that $\det(V) \neq 0$. By the Vandermonde determinant formula (Lemma~\ref{lem: vandermonde determinant}), we have
			\[
			\det(V) = \prod_{0 \le r < l \le n-2} (\gamma_{m,l} - \gamma_{m,r})\,.
			\]
			
			Suppose, for contradiction, that $\gamma_{m,r} = \gamma_{m,l}\,,$ for some $r \neq l$. Then
			$\displaystyle{
				\frac{F_{m,r+1}}{F_{m,r}} = \frac{F_{m,l+1}}{F_{m,l}} }\,,$
			which implies
			\[
			F_{m,l}F_{m,r+1} - F_{m,l+1}F_{m,r} = 0\,. 
			\]
			Hence, by Lemma~\ref{docagnes identity} (D'Ocagne-type identity), 
			\[
			(1-m)^{r+1} F_{m,l-r-1} = 0\,. 
			\]
			However, the left side of this expression is nonzero as $m \geq 2$ and $F_{m,i} \neq 0$ for all $i \geq 0$. Thus, $\gamma_{m,r} \neq \gamma_{m,l}$ for all $0 \leq r < l \leq n-2$, and consequently $\det(V) \neq 0$.
			
			Hence, $\det(W(P)) = \det(V)\det(D)\det(U) \neq 0$. It follows that $\operatorname{rank}(W(P)) = n-1$. 
			Thus, by Lemma~\ref{lem: main eigen. = rank of W(B)}\,, the  zero-divisor graph $\Gamma(R_n)$ has exactly $n-1$ main eigenvalues.  
		\end{proof}
		
		\begin{corollary}
			The determinant of the walk matrix of the quotient matrix $P[m,n]$ is given by
			\[
			det(W(P[m,n])) = \prod_{0 \leq r < l \leq n-2} (\gamma_{m,l} - \gamma_{m,r})\prod_{0 \leq k \leq n-2} C_k\gamma_{m,k}\,,
			\]
			\medskip 
			where 
			$\displaystyle \gamma_{m,k} = \frac{F_{m,k+1}}{F_{m,k}}$ and\, $C_k = F_{m,k}^{\,n}\,$ for all $k\in[n-2]$.
		\end{corollary}

		The following result is an immediate consequence of  Lemma~\ref{main eig of G is an eig of Quotient} and Theorem~\ref{thm:P-rank} and the fact that all $n-1$ eigenvalues of the matrix $P[m,n]$ are distinct and nonzero\cite{sonawane2}. 
		\begin{corollary}\label{cor:main eigenvalues of ZDG are eigenvalue of P}
			The main eigenvalues of the graph $\Gamma(R_n)$ are precisely the eigenvalues of the quotient matrix $P[m,n]$\,, all of which are nonzero.
		\end{corollary}
		As seen in Section~1, the Boolean graph $\Gamma(\mathbb{Z}_2^n)$ has exactly $2n-2$ distinct eigenvalues, of which exactly half are main by the above corollary and the remaining half are non-main. In the next section, we prove that the non-main eigenvalues the graph $\Gamma(R_n)$ are related to the main eigenvalues of  one its subgraphs. 
		
		\section{Main spectrum of a bipartite subgraph of
			\texorpdfstring{$\Gamma(R_n)$}{Gamma(Rn)}}
		\noindent
		In this section, we consider a vertex-induced bipartite subgraph of the zero-divisor graph $ \Gamma(R_n)$ and prove that it has $n-1$ main eigenvalues, using the quotient matrix corresponding to an equitable partition. 
		
		Note that any vertex of the graph $\Gamma(R_n)$ is an $n$-tuple with coordinates belonging to the field $\mathbb{F}_m.$ Sonawane, Kadu and Bosre~\cite{sonawane1} constructed a subgraph of this graph induced by the following two vertex sets: 
		\[
		\begin{aligned}
			X_{*0} &= \{(a_1, a_2, \dots, a_n) \in V(\Gamma(R_n)) \colon a_{n-1} \neq 0,\; a_n = 0\}, \\ \text{and} \quad
			X_{0*} &= \{(a_1, a_2, \dots, a_n) \in V(\Gamma(R_n)) \colon a_{n-1} = 0,\; a_n \neq 0\}.
		\end{aligned}
		\]
		
		
		Let $\Gamma'(R_n)$ be the subgraph of $\Gamma(R_n)$ induced by the vertex set $X_{*0} \cup X_{0*}$. It is easy to see that no two vertices in the same set $X_{*0}$ or $X_{0*}$ are adjacent, since their coordinates corresponding to $*$ are both units, which gives a nonzero product. Hence, this graph is a bipartite graph with bipartition $(X_{*0}, X_{0*})$, and further, it has
		$
		|X_{*0}| + |X_{0*}| = 2|X_{*0}| = 2(m-1)m^{n-2}$ vertices.  
		Then the equitable partition $\mathcal{D} = \{D_1, D_2, \dots, D_{n-1} \}$ of the graph $\Gamma(R_n)$ induces an equitable partition $\mathcal{D'} = \{D'_1, D'_2, \dots, D'_{n-1}\}$ of the graph $\Gamma'(R_n),$ where  $D'_i = D_i \cap \{X_{*0} \cup X_{0*}\}.$  The corresponding quotient matrix, denoted by $Q[m,n],$ is a Pascal-type matrix of order $n-1$ defined as follows.
		\begin{defi}\cite{sonawane1}
			The quotient matrix $Q[m,n]$ of the bipartite subgraph $\Gamma'(R_n)$ with respect to the above equitable partition $\mathcal{D'}$ is an $(n-1) \times (n-1)$ matrix, with $(i,j)$-entry given by
			\[
			Q[m,n](i,j) =
			\begin{cases}
				\displaystyle \binom{i-1}{\,n-j-1\,}(m-1)^{\,n-j}\,, & \text{if } i + j \geq n\,, \\[6pt]
				\quad \quad \quad 0\,, & \text{if } i + j < n\,.
			\end{cases}
			\]
		\end{defi}
		
		For example, 
		\[
		Q[2,4] =
		\begin{bmatrix}
			0 & 0 & 1 \\
			0 & 1 & 1 \\
			1 & 2 & 1
		\end{bmatrix},
		\qquad
		Q[3,4] =
		\begin{bmatrix}
			0 & 0 & 2 \\
			0 & 4 & 2 \\
			8 & 8 & 2
		\end{bmatrix}.
		\]
		Sonawane et al. \cite{sonawane1} proved that if $\lambda$ is an eigenvalue of the matrix $Q[m,n],$ then $-\lambda$ is an eigenvalue of the graph $\Gamma(R_n).$ Moreover, Sonawane \cite{sonawane2} determined the eigenvalues of $Q[m,n]$ and proved that all these eigenvalues are simple and nonzero. 
		
		

		The walk matrix of the quotient matrix $Q = Q[m,n]$ is an $(n-1)\times (n-1)$ matrix $W(Q) = [\mathbf e,\, Q\mathbf e,\, Q^2\mathbf e,\, \dots,\, Q^{n-2}\mathbf e]$. As in case of $W(P),$ we prove that the rank of $W(Q)$ is $n-1.$   We know that the first column of this matrix is $[1,\, 1\, \dots, 1]^T.$ The entries of the next columns are derived below in terms of the generalized Fibonacci numbers.
		
		\begin{lem}\label{lem: entry of WQ}
			Let $Q = Q[m,n]$ be the quotient matrix of the induced bipartite subgraph $\Gamma'(R_n)$ corresponding to the equitable partition $\mathcal{D'}$. Then, for $i \in [n-1]$ and $k \in [n-2]$, the $(i,k+1)$-entry of the walk matrix  $W(Q)$ is given by 
			\[
			[W(Q)]_{i,k+1}= (m-1)^k F_{m,k}^{\,n-2} \, \gamma_{m,k}^{\,i-1}\,,
			\]
			where  $\displaystyle{\gamma_{m,k} = \frac{F_{m,k+1}}{F_{m,k}}}.$
		\end{lem}
		
		\begin{proof}
			We prove the result by induction on $k$. Suppose $ k = 1,$ then the second column of the matrix $W(Q)$ is $Q\mathbf{e}.$  We compute its $i$th entry by evaluating the row sum
			\[
			[Qe]_i =  \sum_{j=n-i}^{n-1} \binom{i-1}{n-j-1}(m-1)^{\,n-j}\,.
			\]
			Letting $r = n - j -1$, we obtain
			$$
			[Q\mathbf{e}]_i 
			= \sum_{r=0}^{i-1} \binom{i-1}{r} (m-1)^{\,r+1} 
			= (m-1)\left(1+(m-1)\right)^{i-1} 
			=(m-1)  \, \gamma_{m,1}^{\,i-1}\,,
			$$
			which matches the required formula as $F_{m,1} = 1$. Hence, the result holds for $k = 1.$ 
			
			Assume that the result holds for all integers up to $k.$ We prove it for $k+1$. The $(k+2)$th column of the matrix $W(Q)$ is  $[Q^{k+1}\mathbf{e}],$ which can be written as $[Q(Q^k\mathbf{e})].$ We compute its $i$th entry as follows:
			
			\begin{align*}
				[W(Q)]_{i,k+2} &=  \left[ Q^{k+1} \mathbf{e} \right]_i = \left[ Q (Q^k \mathbf{e}) \right]_i
				=\sum_{j=n-i}^{n-1} \binom{i-1}{n-j-1}(m-1)^{\,n-j} [Q^k e]_j\\
				&= \sum_{r=0}^{i-1} \binom{i-1}{r} (m-1)^{r+1} \, \left[ Q^{k} \mathbf{e} \right]_{n-r-1}\\
				&= \sum_{r=0}^{i-1} \binom{i-1}{r} (m-1)^{r+1} \left((m-1)^k F_{m,k}^{\,n-2} \, \gamma_{m,k}^{\,n-r-2}
				\right)\\
				&= (m-1)^{k+1}F_{m,k}^{\,n-2} \gamma_{m,k}^{\,n-2}
				\sum_{r=0}^{i-1} \binom{i-1}{r} \left(\frac{m-1}{\gamma_{m,k}}\right)^{r}\\
				&= (m-1)^{k+1} F_{m,k}^{\,n-2} \gamma_{m,k}^{\,n-2}
				\left(1 + \frac{m-1}{\gamma_{m,k}}\right)^{i-1}\\
				& = (m-1)^{k+1} F_{m,k+1}^{\,n-2} \, \gamma_{m,k+1}^{\,i-1}\,.
			\end{align*}
			
			Thus, the result holds for $k+1$, and the proof follows by induction.
		\end{proof}

		\noindent
		We now determine the rank of the walk matrix $W(Q)$ and thus find the number of main eigenvalues of the induced bipartite subgraph $\Gamma'(R_n).$

		\begin{theorem}\label{thm:Q-rank}
			The bipartite graph $\Gamma'(R_n)$ has exactly $n-1$ main eigenvalues.
		\end{theorem}
		\begin{proof}
			Let $ Q = Q[m, n]$ be the quotient matrix corresponding to the equitable partition $\mathcal{D'}$ of the graph $\Gamma'(R_n).$ In view of Lemma \ref{lem: main eigen. = rank of W(B)}\,, it is sufficient to prove that the rank of the walk matrix $W(Q)$ of the quotient matrix $Q$ is $n-1.$ As this walk matrix is an $(n-1)\times(n-1)$ matrix, it is sufficient to prove that its determinant is nonzero. Let $C_k = (m-1)^{k} F_{m,k}^{\,n-2},$ then by Lemma~\ref{lem: entry of WQ} this matrix takes the form
			\[\renewcommand{\arraystretch}{1.2}
			W(Q) =
			\begin{bmatrix}
				1 & C_1 & C_2 & \cdots & C_{n-2} \\
				1 & C_1 \gamma_{m,1} & C_2 \gamma_{m,2} & \cdots & C_{n-2} \gamma_{m,n-2} \\
				1 & C_1 \gamma_{m,1}^2 & C_2 \gamma_{m,2}^2 & \cdots & C_{n-2} \gamma_{m,n-2}^2 \\
				\vdots & \vdots  & \vdots & \ddots & \vdots \\
				1 & C_1 \gamma_{m,1}^{n-2} & C_2 \gamma_{m,2}^{n-2} & \cdots & C_{n-2} \gamma_{m,n-2}^{n-2}
			\end{bmatrix}.
			\]

			Since $\gamma_{m,0}= \frac{F_{m,1}}{F_{m,0}} =1$ and $C_0 =(m-1)^{0} F_{m,0}^{\,n-2} = 1,$ the first column can be replaced by the vector 
			$[C_0,C_0\gamma_{m,0}, C_0\gamma_{m,0}^{\,2},\dots, C_0\gamma_{m,0}^{\,n-2}]^T.$ Hence, 
			
			\[\renewcommand{\arraystretch}{1.2}
			W(Q) =
			\begin{bmatrix}
				C_0 & C_1 & C_2 & \cdots & C_{n-2} \\
				C_0\gamma_{m,0} & C_1 \gamma_{m,1} & C_2 \gamma_{m,2} & \cdots & C_{n-2} \gamma_{m,n-2} \\
				C_0\gamma_{m,0}^2 & C_1 \gamma_{m,1}^2 & C_2 \gamma_{m,2}^2 & \cdots & C_{n-2} \gamma_{m,n-2}^2 \\
				\vdots & \vdots  & \vdots & \ddots & \vdots \\
				C_0\gamma_{m,0}^{n-2} & C_1 \gamma_{m,1}^{n-2} & C_2 \gamma_{m,2}^{n-2} & \cdots & C_{n-2} \gamma_{m,n-2}^{n-2}
			\end{bmatrix}
			= V D,
			\]
			where 
			$$   V =
			\begin{bmatrix}
				1 & 1 & 1 & \cdots & 1 \\
				\gamma_{m,0} & \gamma_{m,1}& \gamma_{m,2} & \cdots & \gamma_{m,n-2} \\
				\gamma_{m,0}^2 & \gamma_{m,1}^2 & \gamma_{m,2}^2 & \cdots & \gamma_{m,n-2}^2 \\
				\vdots & \vdots & \ddots & \vdots \\
				\gamma_{m,0}^{n-2} & \gamma_{m,1}^{n-2} &\gamma_{m,2}^{n-2} & \cdots & \gamma_{m,n-2}^{n-2}
			\end{bmatrix} \quad \text{and} \quad D =\text{diag}(C_0,C_1,C_2, \dots,C_{n-2}).$$ 
			
			Consequently,
			\[
			\det(W(Q)) = \det(V)\det(D).
			\]
			
			Note that $\det(D) \neq 0$, since $F_{m,k} \neq 0$ and $m \ge 2$ imply that $C_k \neq 0$ for all $k$. The Vandermonde matrix $V$ is identical to that arising in the proof of Theorem~\ref{thm:P-rank}. Hence, by the same argument, $\det(V) \neq 0$. Therefore, $\det(W(Q))= \det(V)\det(D) \neq 0. $ It follows that $W(Q)$ is of full rank, that is, $\operatorname{rank}(W(Q)) = n-1$. Thus, by Lemma~\ref{lem: main eigen. = rank of W(B)}\,, the graph $\Gamma'(R_n)$ has exactly $n-1$ main eigenvalues.
		\end{proof}
		
		\begin{corollary}
			The determinant of the walk matrix of the quotient matrix $Q[m,n]$ is given by
			\[
			det(W(Q[m,n])) = \prod_{0\leq r < l \leq n-2} (\gamma_{m,l} - \gamma_{m,r})\prod_{0 \leq k \leq n-2} C_k\,,
			\]
			where
			$
			C_k = (m-1)^k F_{m,k}^{\,n-2} ~~\text{and}~~~\gamma_{m,k} = \frac{F_{m,k+1}}{F_{m,k}}$ ~~for all
			$ k\in \{0, 1, \dots, n-2\}. $ 
		\end{corollary}
		\newpage
		The following result is an immediate consequence of  Lemma~\ref{main eig of G is an eig of Quotient} and Theorem~\ref{thm:Q-rank}. 
		
		\begin{corollary}\label{mainQcor}
			The main eigenvalues of the graph $\Gamma'(R_n)$ are precisely the eigenvalues of the quotient matrix $Q[m, n]$\,, all of which are nonzero.   
		\end{corollary}
		Also, we get the following consequence from  Theorem \ref{spectS}, Corollary \ref{cor:main eigenvalues of ZDG are eigenvalue of P} and Corollary \ref{mainQcor}. 
		\begin{corollary}\label{non-main}
			The nonzero non-main eigenvalues of the graph $\Gamma(R_n)$ are precisely the negatives of the main eigenvalues of its subgraph $\Gamma'(R_n)$.   
		\end{corollary}

		We combine below  Theorem \ref{thm:P-rank} and Theorem \ref{thm:Q-rank}. 
		\begin{theorem}
			For integers $m, s \geq 2$ with $m$ a power of a prime number,  the zero-divisor graph $\Gamma(R_{s+1})$ has exactly $s$ main eigenvalues and its bipartite subgraph $\Gamma'(R_{s+1})$ also has exactly $s$ main eigenvalues.
		\end{theorem}
		Thus, by varying the size $m$ of the field $\mathbb{F}_m$ in the above theorem,  we obtain an infinite family of zero-divisor graphs $\Gamma(R_{s+1})$  with exactly $s$ main eigenvalues. Moreover, their corresponding bipartite subgraphs also form an infinite family with the same property.
		
		\vskip.2cm
		\noindent
		\textbf{An Illustration:} 
		To illustrate the above results, we consider the case $m=2$ and $n=4$. The zero-divisor graph $\Gamma(\mathbb{Z}_2^4)$ and the corresponding induced bipartite subgraph $\Gamma'(\mathbb{Z}_2^4)$ are shown in Figure~\ref{fig:Z24}. The corresponding quotient matrices $P= P[2,4]$,  $Q= Q[2,4]$ and their associated walk matrices are given below
		\[
		\begin{array}{c@{\qquad}c}
			P=
			\begin{bmatrix}
				0 & 0 & 1 \\
				0 & 1 & 2 \\
				1 & 3 & 3
			\end{bmatrix},
			~
			W(P)=
			\begin{bmatrix}
				1 & 1 & 7 \\
				1 & 3 & 17 \\
				1 & 7 & 31
			\end{bmatrix},
			\text{~~and~~}
			Q=
			\begin{bmatrix}
				0 & 0 & 1 \\
				0 & 1 & 1 \\
				1 & 2 & 1
			\end{bmatrix},
			~
			W(Q)=
			\begin{bmatrix}
				1 & 1 & 4 \\
				1 & 2 & 6 \\
				1 & 4 & 9
			\end{bmatrix}
		\end{array}
		\]

		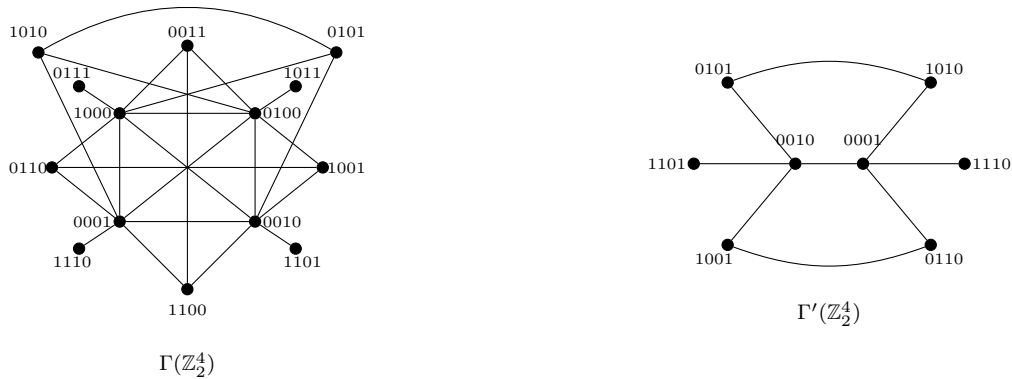
\begin{figure}[!ht]\label{fig:Z24}
			\centering
			
			\begin{minipage}{0.48\textwidth}
				\centering
				\begin{tikzpicture}[
					scale=0.9,
					every node/.style={circle, draw, fill=black, inner sep=1.5pt},
					lbl/.style={draw=none, fill=none, font=\tiny}
					]
					
					\node (1000) at (-1,0.8) {};
					\node (0100) at (1,0.8) {};
					\node (0010) at (1,-0.8) {};
					\node (0001) at (-1,-0.8) {};
					
					\node[lbl] at (-1.4,0.8) {$1000$};
					\node[lbl] at (1.4,0.8) {$0100$};
					\node[lbl] at (1.4,-0.8) {$0010$};
					\node[lbl] at (-1.4,-0.8) {$0001$};
					
					\node (0011) at (0,1.8) {};
					\node (1100) at (0,-1.8) {};
					
					\node[lbl] at (0,2) {$0011$};
					\node[lbl] at (0,-2.1) {$1100$};
					
					\node (1010) at (-2.2,1.7) {};
					\node (0101) at (2.2,1.7) {};
					\node (0110) at (-2,0) {};
					\node (1001) at (2,0) {};
					\node (1110) at (-1.6,-1.2) {};
					\node (1101) at (1.6,-1.2) {};
					\node (0111) at (-1.6,1.2) {};
					\node (1011) at (1.6,1.2) {};
					
					\node[lbl] at (-2.35,2) {$1010$};
					\node[lbl] at (2.35,2) {$0101$};
					
					\node[lbl] at (-2.35,0) {$0110$};
					\node[lbl] at (2.35,0) {$1001$};
					
					\node[lbl] at (-1.7,-1.4) {$1110$};
					\node[lbl] at (1.7,-1.4) {$1101$};
					
					\node[lbl] at (-1.7,1.4) {$0111$};
					\node[lbl] at (1.7,1.4) {$1011$};

					\draw (1000) -- (0100) -- (0010) -- (0001) -- (1000);
					\draw (1000) -- (0010);
					\draw (0100) -- (0001);
					\draw (0011) -- (1100);
					\draw (0110) -- (1001);
					
					\draw (0011) -- (1000);
					\draw (0011) -- (0100);
					\draw (1100) -- (0001);
					\draw (1100) -- (0010);
					
					\draw (1010) -- (0001);
					\draw (1010) -- (0100);
					\draw (0101) -- (0010);
					\draw (0101) -- (1000);
					\draw (0110) -- (1000);
					\draw (0110) -- (0001);
					\draw (1001) -- (0100);
					\draw (1001) -- (0010);
					\draw (1110) -- (0001);
					\draw (1101) -- (0010);
					\draw (0111) -- (1000);
					\draw (1011) -- (0100);
					
					\draw[bend left=30] (1010) to (0101);
					
				\end{tikzpicture}
				
				{\scriptsize $\Gamma(\mathbb{Z}_2^4)$}
			\end{minipage}
			\hspace{0.01\textwidth}
			\begin{minipage}{0.48\textwidth}
				\centering
				\begin{tikzpicture}[
					scale=0.9,
					yscale=1,
					every node/.style={circle, draw, fill=black, inner sep=1.5pt},
					lbl/.style={draw=none, fill=none, font=\tiny}
					]
					
					\node (0010) at (-0.5,0) {};
					\node (0001) at (0.5,0) {};
					
					\node[lbl] at (-0.5,0.35) {$0010$};
					\node[lbl] at (0.5,0.35) {$0001$};
					
					\node (1010) at (-1.5,1.2) {};
					\node (0101) at (1.5,1.2) {};
					
					\node[lbl] at (-1.7,1.4) {$0101$};
					\node[lbl] at (1.7,1.4) {$1010$};
					
					\node (1101) at (-2,0) {};
					\node (1110) at (2,0) {};
					
					\node[lbl] at (-2.4,0) {$1101$};
					\node[lbl] at (2.4,0) {$1110$};
					
					\node (0110) at (-1.5,-1.2) {};
					\node (1001) at (1.5,-1.2) {};
					
					\node[lbl] at (-1.7,-1.4) {$1001$};
					\node[lbl] at (1.7,-1.4) {$0110$};
					
					\draw (0010) -- (0001);
					
					\draw (1010) -- (0010);
					
					\draw (0101) -- (0001);
					
					\draw (0110) -- (0010);
					\draw (1001) -- (0001);
					
					\draw (1101) -- (0010);
					\draw (1110) -- (0001);
					
					\draw[bend left=20] (1010) to (0101);
					\draw[bend right=20] (0110) to (1001);
					
				\end{tikzpicture}
				
				{\scriptsize $\Gamma'(\mathbb{Z}_2^4)$}
			\end{minipage}
			
			\caption{Zero-divisor graph $\Gamma(\mathbb{Z}_2^4)$ and the bipartite subgraph $\Gamma'(\mathbb{Z}_2^4)$.}
			
		\end{figure}

		
		The matrices $W(P)$ and $W(Q)$ each has rank $ 3 = n-1.$ Hence, the graph $\Gamma(\mathbb{Z}_2^4)$ and its bipartite subgraph $\Gamma'(\mathbb{Z}_2^4)$ each has exactly $3$ main eigenvalues, as given below
		
		Main eigenvalues of $\Gamma(\mathbb{Z}_2^4)$  (eigenvalues of $P$):\quad $-1,$\, $\frac{5 -\sqrt{21}}{2},$\, $\frac{5 +\sqrt{21}}{2}$
		
		Main eigenvalues of $\Gamma'(\mathbb{Z}_2^4)$(eigenvalues of $Q$):\quad $-1,$\, $\frac{3 -\sqrt{5}}{2}$,\,  $\frac{3 +\sqrt{5}}{2}$.
		
		\section{Conclusion}
		\noindent
		In this paper, we present two infinite families of graphs arising from algebraic structures that have exactly \(s\) main eigenvalues for any given positive integer \(s\). In particular, we show that the zero-divisor graph of the ring $
		R_n := \mathbb{F}_m \times \mathbb{F}_m \times \cdots \times \mathbb{F}_m $
		(with \(n\) terms) has exactly  \(n\,- 1\) main eigenvalues. Moreover, we prove that the induced bipartite subgraph of \(\Gamma(R_n)\) also has exactly \(n-1\) main eigenvalues, and these are precisely the negatives of the nonzero non-main eigenvalues of the original graph.
		
		\section*{Acknowledgments}
		\noindent
		We thank Sujit Sakharam Damse and  S. A. Katre for their lectures in the lecture series on Matrix Analysis and Positivity at Savitribai Phule Pune University, which were helpful  in improving the exposition of our results.  The first author gratefully acknowledges the Department of Science and Technology (DST), Government of India, for the award of the INSPIRE Fellowship [IF220238].


\begin{thebibliography}{17}
			
			\bibitem{anderson} D. F. Anderson and P. S. Livingston, The zero-divisor graph of a commutative ring. \textit{J. Algebra}, 217 (1999), 434--447.
			
			\bibitem{beck} I. Beck, Coloring of commutative rings. \textit{J. Algebra}, 116 (1988), 208--226.
			
			\bibitem{cvetkovic1} D. Cvetkovi\'{c}, P. Rowlinson and S. K. Simi\'{c}, \textit{An Introduction to the Theory of Graph Spectra}, Cambridge Univ. Press, Cambridge, 2010.
			
			\bibitem{cvetkovic2} D. Cvetkovi\v{c}, The main part of the spectrum, divisors, and switching of graphs. \textit{Publ. Inst. Math. (Beograd)}, 23 (1978), 37--41.
			
			\bibitem{du1} Z. Du, F. Liu, S. Liu and Z. Qin, Graphs with $n-1$ main eigenvalues. \textit{Discrete Math.}, 344(7) (2021), 112397.
			
			\bibitem{du2} Z. Du, L. You and H. Liu, Almost controllable graphs and beyond. \textit{Discrete Math.}, 347(1) (2024), 113743.
			
			\bibitem{franca2026} F. França, A. Brondani and D. Jaume, Graphs with exactly three main eigenvalues. \textit{Discrete Appl. Math.}, 378 (2026), 49--64.
			
			\bibitem{godsil} C. Godsil, Controllable subsets in graphs. \textit{Ann. Comb.}, 16 (2012), 733--744.
			
			\bibitem{ghebleh} M. Ghebleh, S. Al-Yakoob, A. Kanso and D. Stevanovi{\'c}, Graphs having two main eigenvalues and arbitrarily many distinct vertex degrees. \textit{Appl. Math. Comput.}, 495 (2025), 129311.
			
			
			\bibitem{hagos} E. M. Hagos, Some results on graph spectra. \textit{Linear Algebra Appl.}, 356 (2002), 103--111.
			
			\bibitem{harary} F. Harary and A. J. Schwenk, The spectral approach to determining the number of walks in a graph. \textit{Pacific J. Math.}, 80(2) (1979), 443--449.
			
			\bibitem{hou1} Y. Hou and H. Zhou, Trees with exactly two main eigenvalues. \textit{J. Nat. Sci. Hunan Norm. Univ.}, 28 (2005).
			
			\bibitem{hou2} Y. Hou and F. Tian, Unicyclic graphs with exactly two main eigenvalues. \textit{Appl. Math. Lett.}, 19(11) (2006), 1143--1147.
			
			\bibitem{hou3} Y. Hou, Z. Tang and W. C. Shiu, Some results on graphs with exactly two main eigenvalues. \textit{Appl. Math. Lett.}, 25(10) (2012), 1274--1278.
			
			\bibitem{huang} X. Huang, Q. Huang and L. Lu, Construction of graphs with exactly $k$ main eigenvalues. \textit{Linear Algebra Appl.}, 486 (2015), 204--218.
			
			\bibitem{khare} A. Khare, \textit{Matrix Analysis and Entrywise Positivity Preservers}, Cambridge Univ. Press, 2022.
			
			\bibitem{kadu} G. S. Kadu, G. Sonawane and Y. M. Borse, Reciprocal eigenvalue properties using the zeta and Möbius functions. \textit{Linear Algebra Appl.}, 694 (2024), 186--205.
			
			\bibitem{koshy} T. Koshy, \textit{Fibonacci and Lucas Numbers with Applications}, Wiley, New York, 2001.
			
			\bibitem{lagrange1} J. D. LaGrange, Spectra of Boolean graphs and certain matrices of binomial coefficients. \textit{Int. Electron. J. Algebra}, 9 (2011), 78--84.
			
			\bibitem{lagrange2} J. D. LaGrange, Eigenvalues of Boolean graphs and Pascal-type matrices. \textit{Int. Electron. J. Algebra}, 13 (2013), 109--119.
			
			\bibitem{lagrange3} J. D. LaGrange, A combinatorial development of Fibonacci numbers in graph spectra. \textit{Linear Algebra Appl.}, 438(11) (2013), 4335--4347.
			
			\bibitem{lagrange4} J. D. LaGrange, Boolean rings and reciprocal eigenvalue properties. \textit{Linear Algebra Appl.}, 436(7) (2012), 1863--1871.
			
			\bibitem{rowlinson} P. Rowlinson, The main eigenvalues of a graph: a survey. \textit{Appl. Anal. Discrete Math.}, 1 (2007), 445--471.
			
			\bibitem{sonawane1} G. Sonawane, G. S. Kadu and Y. M. Borse, Spectra of zero-divisor graphs of finite reduced rings. \textit{J. Algebra Appl.}, 24(3) (2025), 2550082.
			
			\bibitem{sonawane2} G. Sonawane, Pascal-type matrices and eigenvalues of zero-divisor graphs of reduced rings. \textit{Linear Multilinear Algebra}, (2025), 1--15. 
			
		\end{thebibliography}
	\end{document}